\documentclass[11pt]{amsart}
\usepackage[letterpaper,top=2cm,bottom=2cm,left=2cm,right=2cm]{geometry}

\usepackage[colorlinks=true,linkcolor=blue,urlcolor=blue,citecolor=teal]{hyperref}
\usepackage{graphicx}
\usepackage{caption}
\usepackage{bigstrut}
\usepackage{mathtools}
\usepackage{verbatim}
\usepackage{tikz}
\usepackage[all]{xy}
\usepackage[english]{babel}

\usepackage{cleveref} 
\usepackage{graphicx}
\usepackage{wasysym}
\usepackage{tikz-cd}
\usepackage{nicematrix}
\usepackage{amssymb}

\usetikzlibrary{arrows}
\usepackage{algorithm}
\usepackage{shuffle}
\usepackage[normalem]{ulem}
\usepackage{mathdots} 

\usepackage{graphicx}
\usepackage{subcaption}
\usepackage{array}

\usepackage{comment}

\usepackage{enumitem}

\newtheorem{theorem}{Theorem}
\newtheorem*{theorem*}{Theorem}
\numberwithin{theorem}{section}
\newtheorem{proposition}[theorem]{Proposition}
\newtheorem{lemma}[theorem]{Lemma}
\newtheorem{corollary}[theorem]{Corollary}
\newtheorem*{corollary*}{Corollary}

\theoremstyle{definition}
\newtheorem{definition}[theorem]{Definition}
\newtheorem{notation}[theorem]{Notation}
\newtheorem{remark}[theorem]{Remark}
\newtheorem{example}[theorem]{Example}
\newtheorem*{claim*}{\indent Claim}

\newcommand{\PP}{\mathbb{P}}

\newcommand{\CC}{\mathbb{C}}

\newcommand{\NN}{\mathbb{N}}

\newcommand{\rk}{\mathop{\rm rk}\nolimits}

\newcommand{\sym}{\mathop{\rm Sym}\nolimits}
\newcommand{\spann}{\mathop{\textrm{span}}}
\DeclareMathOperator{\Cat}{Cat}

\newcommand{\B}{\mathcal{B}}
\newcommand{\A}{\mathcal{A}}

\newcommand{\Cc}{\mathcal{C}}

\newcommand{\bfd}{\mathbf{d}}

\newcommand{\bfi}{\mathbf{i}}
\newcommand{\bfj}{\mathbf{j}}
\newcommand{\bft}{\mathbf{t}}

\title{ A new bound on the rank of tensor product of $W$-states}
\author{Stefano Canino \and Alex Casarotti \and Pierpaola Santarsiero}

\newcommand{\Addresses}{{
  \bigskip
  \footnotesize

  \textsc{Stefano Canino, Universit\`a di Trento, Dipartimento di Matematica, Via Sommarive 14, 38123 Povo, Trento, Italy }\par\nopagebreak
  \textit{E-mail address}: \email{stefano.canino@unitn.it}

  \medskip

\textsc{Alex Casarotti, Universit\`a di Ferrara, Dipartimento di Matematica e Informatica, Via Nicol\`o Machiavelli 30, 44121 Ferrara, Italy}\par\nopagebreak
  \textit{E-mail address}: \email{alex.casarotti@unife.it}

  \medskip
  \textsc{Pierpaola Santarsiero, Dipartimento di Ingegneria Industriale e Scienze Matematiche, Universit\`a Politecnica delle Marche, Via Brecce Bianche
I-60131 Ancona, Italy}\par\nopagebreak
  \textit{E-mail address}: \email{p.santarsiero@staff.univpm.it}
}}

\begin{document}

\begin{abstract}
A $W$-state is an order $d$ symmetric tensor of the form $W_d=x^{d-1}y$. We prove that the partially symmetric rank of $W_{d_1}\otimes \cdots \otimes W_{d_k} $ is at most $2^{k-1}(d_1+\cdots +d_k-2k+2)$. The same bound holds for the tensor rank and it is an improvement of $2^k(k-1)$ over the best known bound. Moreover, we provide an explicit partially symmetric decomposition achieving this bound.
\end{abstract}
\maketitle

\section{Introduction}
Computing the rank of a tensor is generally NP-hard, see \cite{Has90, rankNPhard}. For this reason, one often focuses on special families of tensors that serve as meaningful benchmarks either because they originate from applications, or because they exhibit particularly interesting behavior. One such family is the tensor product of generalized $W$-state: $$W_{d_1}\otimes \cdots \otimes W_{d_k}\in \underbrace{\CC^2\otimes \cdots \otimes \CC^2}_{d_1+\cdots +d_k-\text{times}}
$$ 
where each $W$-state $W_d$ is defined as
$$
W_d=\sum_{i=1}^d e^{(1)}_{1}\otimes \cdots \otimes e^{(i-1)}_1\otimes e^{(i)}_2 \otimes e^{(i+1)}_1\otimes \cdots \otimes e^{(d)}_1 \in (\mathbb{C}^2)^{\otimes d}, \quad \spann\{e^{(i)}_1,e^{(i)}_2 \}\cong \CC^2.
$$ 
Focusing for a moment on a single copy, the generalized $W$-state itself is a highly interesting tensor. It is a symmetric tensor and one of the simplest examples illustrating the non-semicontinuity of tensor rank. Moreover, the expression of the general element of the tangential variety $\tau(\nu((\PP^1)^{\times d}))$ of the Segre image $\nu((\PP^1)^{\times d})$ is $W_d$. This is because actually $W_d\in \langle \nu(Z)\rangle$, where $Z\subset (\PP^1)^{\times d}$ is a zero-dimensional scheme of length 2 supported at $\otimes_{i=1}^d e^{(i)}_{1} $, called a 2-jet \cite{alexander1997interpolation}, so actually $W_d$ lies on a line that is tangent to $\nu((\PP^1)^{\times d})$ at $\otimes_{i=1}^d e^{(i)}_{1} $. The $W$-state is a rank-$d$ tensor and it is an example of tensor for which tensor rank and symmetric tensor rank coincide \cite{rankTangentialBB}. It has infinitely many decompositions computing its rank and in particular, one can choose any elementary tensor except for $\otimes_{i=1}^d e^{(i)}_{1}$ to be part of one of its minimal (symmetric) rank decomposition \cite{carlini2017waring, BOS}. 

 From the point of view of complexity theory, $W$-states are related to the so-called Coppersmith-Winogard tensors \cite{coppersmith1987matrix}, which are tensors of interest in the study of the matrix multiplication tensor. In particular, $W$-states are the outer structure of many tensors, including Strassen tensor and truncated polynomial multiplication, used in the so-called \emph{laser method}, a strategy introduced by V. Strassen in \cite{Str86} for computing the complexity of matrix multiplication.

In the three-factors case $W_3$ plays a historical central role in quantum information theory, being one of the two fundamental classes of genuine tripartite entanglement \cite{dur2000three}, and this is where the name \emph{$W$-state} comes from. This makes the $W$-state a structurally very rich object at the intersection of geometry, algebraic complexity, and quantum information.
 
 Given how special $W_d$ is, it is natural to expect that the tensor product of several copies, $W_{d_1}\otimes \cdots \otimes W_{d_k} $, is also quite special. This expectation is reinforced for instance by the role that the tensor product of two copies of $W$-states had in the study of the multiplicativity of tensor rank under tensor product: indeed, \cite{CJZ} proved that the rank of $W_3\otimes W_3$ is strictly less than the naive multiplicative guess $3\cdot 3$. This was the first explicit example showing that the rank of the tensor product of two tensors is not the product of the two ranks. Shortly after, \cite{CF18} proved that the rank of  $W_3\otimes W_3$ is 8. A systematic approach to the strict submulticativity property can be found in \cite{BBGOV}, while, motivated by applications in complexity theory and quantum information theory, the submultiplicativity of the Kronecker tensor product of many copies of $W$-states has been studied for instance in \cite{chen2010tensor, zuiddam2017note}. Some Kronecker products of $W$-states show also other interesting connections. For instance, $W_3\boxtimes W_3$ is a symmetric tensor corresponding to the cubic hypersurface of $\PP^3$ made of a quadric and a tangent hyperplane and, by \cite{Seg42}, this is the only quaternary cubic having maximal rank 7. As a second example, a particular projection of $W_3\boxtimes W_3\boxtimes W_3$ is the so-called Bjorklund-Kaski tensor, which appears in \cite{BK24} and is used for proving that the asymptotic rank conjecture \cite{Str94} and the set cover conjecture \cite{CFKLMMPS15, KT19} cannot be both true.

 Hence, determining the rank of $W_{d_1}\otimes \cdots \otimes W_{d_k}$ becomes a particularly meaningful problem. \cite{BBCG} gave the first bounds on the rank of the tensor product of an arbitrary number of $W$-states 
 and these remain the best known results to date, with the notable exception of the case $k=2$ where the bound was later improved in \cite[Theorem 1.8(i)]{gal}. 

  In this paper, we make progress along this direction by providing a new sharp upper bound on the rank of tensor products of generalized $W$-states $W_{d_1}\otimes \cdots \otimes W_{d_k}\in \bigotimes_{i=1}^k(\CC^{2})^{\otimes d_i}$. We remark that $W_{d_1}\otimes \cdots \otimes W_{d_k}$ is a partially symmetric tensor and for a partially symmetric tensor $T\in \sym^{d_1}\CC^{n_1}\otimes \cdots \otimes \sym^{d_k}\CC^{n_k}$ the partially symmetric rank of $T$, denoted as $\rk_{\bfd}(T)$, is the minimum integer $r$ such that $T=\sum_{i=1}^r (v^{(1)}_i)^{\otimes d_1}\otimes \cdots \otimes (v^{(k)}_i)^{\otimes d_k}$, where $v^{(j)}_i\in \CC^{n_j}$. Our main result is the following.

\begin{theorem}\label{theorem: bound rango}
For every $k\geq 2$ and every $d_1,\dots,d_k\geq 3$ we have
$$
\rk_{\bfd}(W_{d_1}\otimes \cdots \otimes W_{d_k})\leq  2^{k-1}\left(d_1+\cdots+d_k-2k+2\right).
$$
\end{theorem}
\Cref{theorem: bound rango} improves \cite[Theorem 3.6]{BBCG} by $2^k(k-1)$. For $k=2$ it coincides with the bound provided in \cite[Theorem 1.8(i)]{gal} and when all $d_i=3$ it matches the bound of \cite[Theorem 3.3]{BBCG}. The bound of \Cref{theorem: bound rango} is sharp as for $k=2$ and $d_1=d_2=3$ it gives rank 8.

 In addition to providing a theoretical upper bound on the tensor rank, \Cref{theorem: bound rango} is constructive. Indeed, the proof yields an explicit decomposition of length equal to the bound that can be achieved by computing the rank decomposition of essentially one binary form (see \Cref{algo}). This decreases a lot the computational cost for computing a partially symmetric decomposition of $W_{d_1}\otimes\cdots\otimes W_{d_k}$; see the forthcoming \Cref{remark: commento algo}.
 
 As an immediate consequence the same bound holds also for the rank of tensor product of $W$-states. \begin{corollary}\label{corollary: bound rango vero}
Let $k\geq 2$, fix positive integers $d_i\geq 3$ for $i=1,\dots,k$. 
The rank of $W_{d_1}\otimes \cdots \otimes W_{d_k}\in \bigotimes_{i=1}^k(\CC^{2})^{\otimes d_i}$ is at most $2^{k-1}(d_1+\dots+d_k-2k+2)$.
\end{corollary}

Our methods to achieve this result are geometric and, as in \cite[Theorem 3.6]{BBCG}, they rely on finding a curve in $\mathcal{C}\subset (\PP^1)^{\times k}$ such that the span of the image of $\Cc$ under the Segre-Veronese embedding contains $W_{d_1}\otimes \cdots \otimes W_{d_k}$. 

\subsection*{Outline of the paper}  We set up our notation and give standard preliminaries in \Cref{section: preliminaries}, while \Cref{section: solving the system} contains the main technical results to get our bounds. We start \Cref{section: computing the bounds} by briefly reviewing the theory of catalecticants needed to compute the rank of a family of binary homogeneous polynomials (\Cref{prop: rango TJ}). Then, we prove our main \Cref{theorem: bound rango} and in \Cref{algo} we give a recipe to explicitly compute a partially symmetric decomposition of length given by our bound. \Cref{theorem: bound rango bordo} computes the border partially symmetric rank of $W_{d_1}\otimes \cdots \otimes W_{d_k}$. We conclude with a detailed comment on the method in \Cref{subsection: commenti sul bound}.

\subsection*{Acknowledgements}
The authors are grateful to Fulvio Gesmundo, Alessandro Gimigliano and Joachim Jelisiejew for interesting conversations on the topic.

S. Canino has been funded by the Italian Ministry of
University and Research in the framework of the Call for Proposals for
scrolling of final rankings of the PRIN 2022 call - Protocol no.
2022NBN7TL.
A. Casarotti has been funded by the European Union under the project NextGenerationEU. PRIN 2022, CUP: F53D23002600006.
P. Santarsiero was supported by the European Union under NextGenerationEU. PRIN 2022, Prot. 2022E2Z4AK and PRIN 2022 SC-CUP: I53C24002240006.

\section{Preliminary notions}\label{section: preliminaries}
We work over $\CC$. In this part we set up our notation and recall preliminary notions used throughout the text. We refer for instance to \cite{Lands} and \cite{guida} for a more detailed account.

Let $X\subset \PP^N$ be an irreducible non-degenerate projective variety. We denote by $$\sigma_r^\circ(X)=\bigcup_{p_1,\dots,p_r\in X}\langle p_1,\dots,p_r\rangle$$ and the $r$-th secant variety of $X$, denoted as $\sigma_r(X)$, is the Zariski closure in $\PP^N$ of $\sigma_r^\circ(X)$.  Secant varieties are auxiliary varieties needed to study the concept of $X$-rank with respect to a variety $X$.
\begin{definition}
    Let $X\subset \PP^N$ be irreducible and non-degenerate. The \emph{$X$-rank} of a $T\in \PP^N$, denoted as $\rk_X(T)$, is the minimum integer $r$ such that $T\in \langle p_1,\dots,p_r \rangle$, for $p_1,\dots,p_r\in X$. The border $X$-rank of $ T$, denoted as $\underline{\rk}_X(T)$ is the minimum integer $r$ such that $T\in \sigma_r(X)$.
\end{definition}
Fix a positive integer $k\geq 1$ and let us consider the multiprojective space $\PP^1\times \cdots \times \PP^1$ given by $k$ copies of $\PP^1$ with coordinates $[x_{1,0},x_{1,1};\ldots;x_{k,0},x_{k,1} ]$. Let $\bfd=(d_1,\dots,d_k)$ for positive integers $d_1,\dots,d_k$, and denote by
$$sv_{\boldsymbol{d}}:\PP^1\times\dots\times\PP^1\to \PP^{N_{\boldsymbol{d}}}=\PP(\sym^{d_1}\CC^2\otimes \cdots \otimes \sym^{d_k}\CC^2)$$
the Segre-Veronese embedding of $\PP^1\times\dots\times\PP^1$ with multidegree ${\bf d}$, where $N_{\boldsymbol{d}}=\prod_{i=1}^k(d_i+1)-1$.
We will use the following notation.
\begin{itemize}
    \item The image of $sv_{\boldsymbol{d}}$ is the Segre-Veronese variety $SV^\bfd\subset \PP^{N_{\bf d}}$ of multidegree $\bfd$ and the $SV^\bfd$-rank of a $T\in \PP^{N_{\bfd}}$ is called the \emph{partially symmetric rank} of $T$ and is denoted as $\rk_{\bfd}(T)$. The border $X$-rank is denoted as $\underline{\rk}_{\bf d}(\cdot)$.
    \item When $d_i=1$ for all $i=1,\dots,k$ the corresponding variety is a Segre variety $\mathrm{Seg}_{1^k}\subset \PP^{N_{\bfd}}$ where $N_{\bf d}=2^k-1$ and the $\mathrm{Seg}_{1^k}$-rank of a $T\in \PP^{N_\bfd}$ is called the \emph{rank} of $T$ and is denoted as $\rk(T)$. 
    \item Lastly, if $k=1$ and $d_1=d$ for some positive integer $d$, then $sv_{\boldsymbol{d}}((\PP^1)^{\times k})$ is the degree $d$ rational normal curve $\mathcal{C}_{d}\subset \PP^{N_{\bfd}}$, where $N_{\bf d}=\binom{n+d}{d}-1$ and the $\mathcal{C}_{d}$-rank of a tensor $T\in \PP^{N_{\bfd}}$ is called the \emph{symmetric (or Waring) rank} of $T$ and is denoted as $\rk_{s}(T)$, while its border rank is denoted as $\underline{\rk}_s(T)$.
\end{itemize}
Clearly, for a $T\in \PP(\sym^{d_1}\CC^2\otimes \cdots \otimes \sym^{d_k}\CC^2)$ we have $ \rk(T)\leq \rk_{\bfd }(T)$.

In the following we will work on the multiprojective $(\PP^1)^{\times k}$ with $k\geq 2$. We consider the $i$-th $\mathbb{P}^1$ as a space of binary linear forms  and we consider $\mathbb{P}^{N_{\mathbf{d}}}$ as the space of multilinear forms of multidegree $\bfd$ and total degree $d\coloneqq d_1+\dots+d_k$. In $\mathbb{P}^{N_{\mathbf{d}}}$ we use coordinates 
$$z_{\bfi}=z_{i_1,\dots,i_k}=\prod_{r=1}^k\binom{d_r}{i_r}x_{1,0}^{i_1}x_{1,1}^{d_1-i_1}\otimes \cdots \otimes x_{k,0}^{i_k}x_{k,1}^{d_k-i_k}$$ ordered via graded lexicographic order. 

\begin{notation}
The coordinates of $\PP^{N_{\bf d}}$ are labelled via the set
$\A_{\bfd}\coloneqq\{(i_1,\dots,i_k)\;|\; 0\leq i_r\leq d_r\}$.
Moreover, for any $s\in\NN$ with $0\leq s\leq d$ we set
$\A_{\bfd,s}\coloneqq\{(i_1,\dots,i_k)\in\A_{\bf d}\;|\; i_1+\dots+i_k=s\}.$
In $\A_\bfd$ we use the same ordering of the coordinates $z_{i_1,\dots,i_k}$. 

Given non negative integers $a,b$ with $a\leq b$, we denote by $[a,b]=\{a,a+1,\dots,b \}$.
\end{notation}

Since it will be useful later, we define a simple operator $\Delta(\cdot)_{r_1,r_2}$ depending on two integers $r_1,r_2\in [1,k]$, which acts on a $k$-tuple of non negative integers by decreasing its $r_1$-th entry by one and increasing its $r_2$-th entry by one. 
\begin{definition}\label{definition: standard shift}
Let $k\in\NN$, $\bfd\in\NN^k$. For any $r_1,r_2\in[1,k]$ and for any $k$-tuple $\bft=(t_1,\dots,t_k)\in\A_{\bfd}$ with $t_{r_1}\geq 1$ and $t_{r_2}\leq d_{r_2}-1$, we set 
$$\Delta_{r_1,r_2}(\bft)=\begin{cases}
(t_1,\dots,t_{r_1-1},t_{r_1}-1,t_{r_1+1},\dots,t_{r_2-1},t_{r_2}+1,t_{r_2+1},\dots,t_k),& \text{ if $r_1\leq r_2$}\\
(t_1,\dots,t_{r_2-1},t_{r_2}+1,t_{r_2+1},\dots,t_{r_1-1},t_{r_1}-1,t_{r_1+1},\dots,t_k),& \text{ if $r_1>r_2$}
\end{cases}.$$
Given $\bft,\bft'\in\A_{\bfd,s}$, we say that \emph{$\bft$ and $\bft'$ differ by a standard shift} if there exist $r_1,r_2\in [1,k]$ such that $\bft=\Delta_{r_1,r_2}(\bft')$.
\end{definition}

Given $\bfd\in\NN^k_{\geq3}$, we want to associate to any $J\subseteq [2,k]$ 
a rational normal curve of degree $d=d_1+\cdots+d_k$ in $\PP^{N_\bfd}$.

\begin{definition}
Let $\bfd\in\NN^k_{\geq 3}$ and $d=d_1+\cdots+d_k$. For any $J\subseteq[2,k]$ let $L^{J}\subset \PP^1\times\dots\times \PP^1 $ be the $(1,\dots,1)$-curve defined by the ideal
$$
I^{J}=(x_{1,0}x_{2,1}-\varepsilon_2^J x_{1,1}x_{2,0},\dots,x_{1,0}x_{k,1}-\varepsilon_k^J x_{1,1}x_{k,0}),\text{ where }\varepsilon_r^J=\begin{cases}
     1, & \hbox{if }r\notin J\\
      -1, & \hbox{otherwise}.
\end{cases}
$$
The \emph{degree $d$ rational normal curve associated to $J$} is $\mathcal{C}_{\bf d}^{J}\coloneqq sv_{\boldsymbol{d}}(L^{J})$.
\end{definition}
Note that the set $J$ indicates in which generators of $I^J$ we have a plus sign. If $J=\emptyset$, we write $L$, resp. $I$ and $\mathcal{C}_{\bf d}$, instead of $L^{J}$, resp $I^{J}$ and $\mathcal{C}_{\bf d}^{J}$.

\begin{remark}\label{remark: W-state sta nel quadrello}
Let us consider the affine chart $\mathbb{A}^k\subset\PP^1\times\dots\times\PP^1$ given by $\{x_{1,1}\neq0,\dots,x_{k,1}\neq 0\}$ and let us use in $\mathbb{A}^k$ the affine coordinates
$$x_1=\frac{x_{1,0}}{x_{1,1}},\dots,x_k=\frac{x_{k,0}}{x_{k,1}}.$$
In this setting, the curves $L^{J}$ are lines whose parametric equations are of the form
$$
\begin{cases}
x_1=s\\
x_2=\varepsilon_2^Js\\
\vdots\\
x_k=\varepsilon_k^Js
\end{cases}.$$
As a consequence, the affine ideal defining union of such lines is
$K=(x_2^2-x_1^2,\dots,x_k^2-x_1^2)$
and, in particular, we have $K\subset (x_1^2,\dots,x_k^2).$

\begin{notation}
Let $\PP^1\times \cdots \times \PP^1$ be the multiprojective space given by $k\geq 2$ copies of $\PP^1$. Let $Z\subset \PP^1$ be the length two zero-dimensional scheme supported at $[0:1] $. We denote by $Z_k=(Z)^{\times k}\subset (\PP^1)^{\times k}$ the length $2^k$ zero-dimensional scheme supported at $[0,1;\dots;0,1]$.
\end{notation}
The scheme $Z_k$ is an example of a $2$-symmetric zero-dimensional scheme \cite[Definition 2.1]{CCGI25}, i.e. it possesses the property that any smooth curve of $\PP^1\times\cdots\times\PP^1$ passing through its support intersects $Z_k$ in a 0-dimensional scheme of length $2$ \cite[Proposition 2.18]{CCGI25}. More precisely, $Z_k$ is a $2$-hypercube, i.e. it has maximal length among the schemes possessing the aforementioned property of symmetry. By \Cref{remark: W-state sta nel quadrello}, we have that $Z_k\subset \PP^1\times\dots\times\PP^1$ is contained in the union of the $L^{J}$'s. Hence, by \cite[Lemma 2.1]{BBCG}, we have that $W_{d_1}\otimes\dots\otimes W_{d_k}$ is in the span of the union of the curves $\mathcal{C}_{\bfd}^{J}\subset \PP^{N_{\bf d}}$.
\end{remark}

In the following example we see what happens when $k=2$ and $d_1=d_2=3$ and this case will serve as a running example throughout the whole discussion.

\begin{example}\label{example: step1}
Let us take $k=2$ and $\bfd=(3,3)$ so that $d=3+3=6$. In this case, we have only two possibilities for $J$: either $J=\emptyset$ or $J=\{2\}$. The $(1,1)$-curves $L$ and $L^{\{2\}}$ in $\PP^1\times\PP^1$ are respectively defined by the ideals
$$I=(x_{1,0}x_{2,1}-x_{1,1}x_{2,0}),\qquad I^{\{2\}}=(x_{1,0}x_{2,1}+x_{1,1}x_{2,0}),$$
and the two corresponding rational normal sestics are $\mathcal C_{(3,3)}\coloneqq sv_{(3,3)}(L),$ $\mathcal C_{(3,3)}^{\{2\}}\coloneqq sv_{(3,3)}(L^{\{2\}})$.
\end{example}

\section{Solving the system}\label{section: solving the system}

So far we established that $$ W_{d_1}\otimes \cdots \otimes W_{d_k}\in  \spann\left\{ \cup_{J\subseteq [2,k]} \Cc_{\bf d}^J \right\}\subseteq \PP(\sym^{d_1}\CC^2\otimes \cdots \otimes \sym^{d_k} \CC^2).$$
The purpose of this section is to explicitly write $W_{d_1}\otimes \cdots \otimes W_{d_k}$ in coordinates exploiting the above containment.

Let us start by providing equations for the $d$-dimensional linear spaces  $\mathrm{span}(\mathcal{C}_{\bfd}^{J} )$ containing the rational normal curve $\mathcal{C}_{\bfd}^{J}\subset \PP^{N_{\boldsymbol{d}}}$. The equations will be of the form 
  $   z_{\mathbf{i}}=\pm z_{\mathbf{i}' }$,
where $\mathbf{i}$ and $\mathbf{i}'$ differ by a standard shift (see \Cref{definition: standard shift}). The sign is determined by $J$: each time the shift involves an element of $J$ the minus sign is applied.

\begin{lemma}\label{lemma: eq lineari crn} Let $k\geq 2$ and $ J\subseteq [2,k]$. Cartesian equations describing $\mathrm{span}(\mathcal{C}^{J}_{\bfd})$ are
 $$
 z_{\bfi} = \varepsilon^J_{r} \cdot z_{\Delta_{1,r}(\bfi)},
 $$
 where $2\leq r\leq k$, $ i_1\geq 1,   i_r\leq d_r-1$, and  if $r\notin J$ then $\varepsilon^J_{r}=1$, otherwise $\varepsilon^J_{r}= -1$.
\end{lemma}
\begin{proof}
Let $Q\in \mathcal{C}_{\bfd}^{J} $, so there exists $P\in L^{J}$ with $P=[x_{1,0},x_{1,1};\dots;x_{k,0},x_{k,1}]$ such that $sv_{\mathbf{d}}(P)=Q$. In particular, $z_{\bfi}(Q)=x_{1,0}^{i_1}x_{1,1}^{d_1-i_1}\cdots x_{k,0}^{i_k}x_{k,1}^{d_k-i_k} $. Since $P\in L^{J}$ we have that $x_{1,0}x_{r,1}=\varepsilon^J_{r} x_{1,1}x_{r,0}$ for all $r=2,\dots,k$. Hence, if $i_1\geq 1$ and $i_r\leq d_r-1$ then we have
\begin{align}\label{eq: le zeta}
z_{\bfi}(Q)&=x_{1,0}^{i_1}x_{1,1}^{d_1-i_1}\cdots x_{r,0}^{i_r}x_{r,1}^{d_r-i_r}\cdots x_{k,0}^{i_k}x_{k,1}^{d_k-i_k} \\
&=\varepsilon^J_r x_{1,0}^{i_1-1}x_{1,1}^{d_1-(i_1-1)}\cdots x_{r,0}^{i_r+1}x_{r,1}^{d_r-(i_r+1)}\cdots x_{k,0}^{i_k}x_{k,1}^{d_k-i_k} 
=\varepsilon^J_r z_{\Delta_{1,r}(\bfi)}(Q). \nonumber
\end{align}
Note that the previous equations imply that, if $\bfi,\bfj\in\A_{\bfd,s}$, then  $z_{\bfi}(Q)=(-1)^{a}z_{\bfj}(Q)$ for a certain $a\in \mathbb{N}$. Hence, there are at most $d+1$ free parameters defining the linear space given by \cref{eq: le zeta} and containing $\mathrm{span}(\mathcal{C}_{\bfd}^J)$, and this concludes the proof.
\end{proof}

\begin{remark}\label{remark: equations relating z with the same degree}
Given $J\subseteq [2,k]$ and two $k$-ples $\bfi,\bfj\in\A_{\bfd,s}$, by \Cref{lemma: eq lineari crn} we have that $z_{\bfi}=(-1)^az_{\bfj}$ is an equation of $\spann(\mathcal{C}_{\bfd}^J)$ for a suitable $a\in \{0,1\}$, and \Cref{lemma: eq lineari crn} allows us to compute $a$. For instance, if we consider a $k$-ple $\bfi$ with $i_1\geq 1$ and $\Delta_{r_1,r_2}(\bfi)$, since $\Delta_{r_1,r_2}(\bfi)=\Delta_{r_1,1}\Delta_{1,r_2}(\bfi)$, using twice \Cref{lemma: eq lineari crn} we find
$$z_{\Delta_{r_1,r_2}(\bfi)}=z_{\Delta_{r_1,1}\Delta_{1,r_2}(\bfi)}=\varepsilon_{r_2}z_{\Delta_{r_1,1}(\bfi)}=\varepsilon_{r_1}\varepsilon_{r_2}z_{\bfi}.$$
Note that if $i_1=0$, it is enough to write $\Delta_{r_1,r_2}(\bfi)=\Delta_{1,r_2}\Delta_{r_1,1}(\bfi)$.
In general, it is possible to repeatedly use \Cref{lemma: eq lineari crn} to find the only $a\in\{0,1\}$ such that $z_{\bfi}=(-1)^az_{\bfj}$
is an equation of $\spann(\mathcal{C}_{\bfd}^J)$.
\end{remark}

\begin{remark}\label{remark: free parameters rnc}
Given $J\subseteq[2,k]$, \Cref{lemma: eq lineari crn} tells us that the generic element $T^J\in\spann(\Cc_{\bfd}^J)$ depends on $d+1$ free parameters and there is one free parameter $\alpha_s^J$ for each set of coordinates $\{z_{\bfi}\;|\;\bfi\in\A_{\bfd,s}\}$.
Moreover, by \Cref{remark: equations relating z with the same degree}, we have
$$T^J=(\varepsilon_{\bfi}^J\alpha_s^J)_{\substack{0\leq s\leq d\\ \bfi\in\A_{\bfd,s}}}$$
for suitable $\varepsilon_{\bfi}^J\in\{1,-1\}$. Note that, for any $0\leq s\leq d$ there is a free choice of one (and only one) $\varepsilon_{\bfi}^J$ to be $1$ or $-1$, where $\bfi\in\A_{\bfd,s}$. After we have chosen one $\varepsilon_{\bfi}^J$ with $\bfi\in\A_{\bfd,s}$, then $\varepsilon_{\bfj}^J$ is determined by $J$ for all $\bfj\in \A_{\bfd,s}$. 
\end{remark}

Now, we want to use \Cref{remark: equations relating z with the same degree} to explicitly compute the $\varepsilon_{\bfi}^J$'s.

\begin{lemma}\label{lemma: determine epsilon}
Let $s\in[0,d]$, $J\subseteq[2,k]$, $\bft=(t_1,\dots,t_k)\in\A_{\bfd,s}$ such that $\varepsilon^J_{\bft}=1$ and $\bfi\in\A_{\bfd,s}$. Let
$$N_{\bfi}=\{i_j\;|\;t_j-i_j \text{ is odd},\, j\geq 2\}.$$ Then
$\varepsilon_{\bfi}^J$ is given by $\varepsilon_{\bfi}^J=(-1)^{|N_{\bfi}\cap J|}$. 
\end{lemma}
\begin{proof}
Let
$$L_{\bfi}=\left\{l_1,\dots,l_p\in[2,k]\;|\; i_{l_e}<t_{l_e}\right\},\quad M_{\bfi}=\left\{m_1,\dots,m_q\in[2,k]\;|\; i_{m_f}>t_{m_f}\right\}.$$
Up to the order of the operations (there exists always at least one), we have
$$\bfi=\Delta_{1,l_1}^{t_{l_{1}}-i_{l_{1}}}\dots\Delta_{1,l_p}^{t_{l_{p}}-i_{l_{p}}}\Delta_{m_1,1}^{i_{m_{1}}-t_{m_{1}}}\dots\Delta_{m_b,1}^{i_{m_{q}}-t_{m_q}}(\bft), $$
and thus
$$\varepsilon_{\bfi}^J=\prod_{e=1}^p\varepsilon_{l_e}^{t_{l_{e}}-i_{l_{e}}}\prod_{f=1}^q\varepsilon_{m_f}^{i_{m_{f}}-t_{m_f}}.$$
If we set
$$L_{\bfi}'=\left\{l_e\in L_{\bfi}\;|\;t_{l_e}-i_{l_{e}}\text{ is odd}\right\},\quad M_{\bfi}'=\left\{m_f\in M_{\bfi}\;|\;t_{m_f}-i_{m_{f}}\text{ is odd}\right\},$$
then 
$$\varepsilon_{\bfi}^J=\prod_{l_e\in L_{\bfi}'}\varepsilon^J_{l_e}\prod_{m_f\in M_{\bfi}'}\varepsilon^J_{m_f}=\prod_{l_e\in L_{\bfi}'}(-1)^{|\{l_e\}\cap J|}\prod_{m_f\in M_{\bfi}'}(-1)^{|\{m_f\}\cap J|}=(-1)^{|L_{\bfi}'\cap J|+|M_{\bfi}'\cap J|.}$$
To get the statement, it is enough to note that $L'_{\bfi}$ and $M_{\bfi}'$ are disjoint.
\end{proof}

Let us continue \Cref{example: step1} and let us see how to explicitly write the generic element of $\spann(\Cc_d^J)$ thanks to \Cref{lemma: determine epsilon}.

\begin{example}\label{example: step2}
We take  $k=2$ and $\bfd=(3,3)$ and we want to write explicitly the general elements $T\in\spann(\mathcal C_{(3,3)})$ and $T^{\{2\}}\in\spann(\mathcal C_{3,3}^{\{2\}})$ using the results just described. By \Cref{remark: free parameters rnc}, we know that the coordinates of $T$ and $T^{\{2\}}$ are of the form
$$T=(\varepsilon_{\bfi}\alpha_s)_{\substack{0\leq s\leq 6\\ \bfi\in\A_{(3,3),s}}},\qquad T^{\{2\}}=(\varepsilon_{\bfi}^{\{2\}}\alpha_s^{\{2\}})_{\substack{0\leq s\leq 6\\ \bfi\in\A_{(3,3),s}}}.$$
Moreover, for any $0\leq s\leq 6$, we have to choose $\bfi,\bfi'\in\mathcal A_{(3,3),s}$ and set $\varepsilon_\bfi$ and $\varepsilon_{\bfi'}^{\{2\}}$ equal to 1. In principle, $\bfi$ and $\bfi'$ could be different but, for the sake of simplicity, here, and in the rest of the paper, we choose the same $\bfi$ for any $J$. For $s=1,2,3,5,6$ we choose the first $\bfi\in\mathcal A_{(3,3),s}$ with respect to the lexicographic order, while for $s=d-k=4$ we choose $\bfi=(2,2)$ which is the multi-index corresponding to $W_3\otimes W_3$. For this choice of the $\bfi$'s, we set $\varepsilon_\bfi^J=1$ for $J=\emptyset,\{2\}$. More precisely, we have
$$\varepsilon_{(0,0)}^J=\varepsilon_{(1,0)}^J=\varepsilon_{(2,0)}^J=\varepsilon_{(3,0)}^J=\varepsilon_{(2,2)}^J=\varepsilon_{(3,2)}^J=\varepsilon_{3,3}^J=1,\quad\text{for $J=\emptyset,\{2\}$}.$$
Applying \Cref{lemma: determine epsilon} with this choice, we find that the coordinates of $T$ and $T^{\{2\}}$ are
\begin{gather*}
T=(\alpha_0,\alpha_1,\alpha_1,\alpha_2,\alpha_2,\alpha_2,\alpha_3,\alpha_3,\alpha_3,\alpha_3,\alpha_4,\alpha_4,\alpha_4,\alpha_5,\alpha_5,\alpha_6)\\
T^{\{2\}}=(\alpha_0^{\{2\}},\alpha_1^{\{2\}},-\alpha_1^{\{2\}},\alpha_2^{\{2\}},-\alpha_2^{\{2\}},\alpha_2^{\{2\}},\alpha_3^{\{2\}},-\alpha_3^{\{2\}},\alpha_3^{\{2\}},-\alpha_3^{\{2\}},-\alpha_4^{\{2\}},\alpha_4^{\{2\}},-\alpha_4^{\{2\}},\alpha_5^{\{2\}},-\alpha_5^{\{2\}},\alpha_6^{\{2\}}).
\end{gather*}
\end{example}
The following lemmata are technical results needed later. In this first lemma we are interested in  comparing cardinalities of subsets of $[2,k]$ whose intersection with a given subset $N$ of $[2,k]$ is either even or odd.
\begin{lemma}\label{lemma: even and odd}
Let $k\geq 2$ and $N\subseteq[2,k]$. Then
\begin{enumerate}[label={(\arabic*)}]
    \item \label{case 1 lemma even odd} if $N\neq\emptyset$, then
    $$|\{J\subseteq [2,k] \;|\;|J\cap N|\text{ is even}\}|=|\{J\subseteq [2,k] \;|\;|J\cap N|\text{ is odd}\}|=2^{k-2};$$
    \item \label{case 2 lemma even odd} if $N\neq[2,k]$, then 
    $$|\{J\subseteq [2,k]\;|\; |J|+|J\cap N|\text{ is even}\}|=|\{J\subseteq [2,k]\;|\; |J|+|J\cap N|\text{ is odd}\}|=2^{k-2}.$$ 
\end{enumerate}

\end{lemma}
\begin{proof}
Since $|2^{[2,k]}|=2^{k-1}$, in both points \ref{case 1 lemma even odd} and \ref{case 2 lemma even odd} it is enough to prove one of the equalities. A subset $J\subseteq[2,k]$ can be uniquely written as $J=J_1\cup J_2$ with $J_1\subseteq N$ and $J_2\subseteq[2,k]\setminus J_1$, so that $|J|=|J_1|+|J_2|$ and $|J\cap N|=|J_1|$. It follows that
$$|\{J\subseteq [2,k]\;|\;|J\cap N|\text{ is even}\}|=2^{k-1-|N|}2^{|N|-1}=2^{k-2}.$$
Moreover, $|J|$ and $|J\cap N|$ have the same parity if and only if $|J_2|$ is even. Hence,
\begin{equation*}
|\{J\subseteq [2,k]\;|\; |J|+|J\cap N|\text{ is even}\}|=2^{|N|}2^{k-1-|N|-1}=2^{k-2}. \qedhere
\end{equation*}
\end{proof}

We want to use the previous lemma to have a closer look at the cardinality of the set $\B_{\bfd,s}$ made of all the distinct $(\varepsilon^{\{ 2\}}_{\bfi},\dots,\varepsilon^{\{ k\}}_{\bfi})$ as $\bfi$ varies in $\A_{\bfd, s}$ for a fixed $s\in [0,d]$, and we are especially interested in those $\B_{\bfd,s}$ whose cardinality is $2^{k-1}$. Before proceeding, let us make an example.
\begin{example}
Let $k=3$, $\bfd=(3,3,3)$. Let $s=0$, so that $\A_{\bfd,s}=\{(0,0,0) \}$ and $\B_{\bfd,s}$ defined as $$\B_{\bfd,s}=\{ (\varepsilon_{\bfi}^{\{2\}},\dots,\varepsilon_{\bfi}^{\{k\}})\;|\;\bfi\in\A_{\bfd,s} \}$$ is simply the singleton $\{(\varepsilon_{0,0,0}^{\{2\}},\varepsilon_{0,0,0}^{\{ 3\}}) \}$. Let now $s=1$ so that $\A_{\bfd,s}=\{ (1,0,0),(0,1,0),(0,0,1)\}$ and fix $\varepsilon^J_{1,0,0}=1$ for $J=\{2\},\{3\}$. For $\bfi=(1,0,0) $ we have $(\varepsilon^{\{ 2\}}_{\bfi},\varepsilon^{\{ 3\}}_{\bfi})=(1,1)$, for $\bfi=(0,1,0) $ we have $(\varepsilon^{\{ 2\}}_{\bfi},\varepsilon^{\{ 3\}}_{\bfi})=(-1,1)$, and the last couple will be $(1,-1)$. 
So $|\B_{\bfd,1}|=3$. Now take $s=2$ and fix $\varepsilon^J_{2,0,0}=1$ for $J=\{2\},\{3\}$.
In this case
$$
(\varepsilon^{\{ 2\}}_{\bfi},\varepsilon^{\{ 3\}}_{\bfi})=\begin{cases}
    (1,1), & \text{for }\bfi=(2,0,0),(0,2,0),(0,0,2),\\
    (-1,1), & \text{for }\bfi=(1,1,0),\\
    (1,-1), & \text{for } \bfi=(1,0,1),\\
    (-1,-1), & \text{for } \bfi=(0,1,1).
\end{cases}
$$
Hence  $|\B_{\bfd,2}|=4$ and $s=2$ is the first integer in $\{0,\dots,9\}$ for which this happens. 
\end{example}
The behaviour just analyzed is not a coincidence, in fact we prove the following.
\begin{lemma}\label{lemma: hadamard if greater than k-2}
Let $k\geq 2$, $\bfd\in\NN^k_{\geq 3}$, choose $\bft\in\A_{\bfd,s}$ such that $\varepsilon_{\bft}^J=1$ for any $J\subseteq[2,k]$ and set  \linebreak $\B_{\bfd,s}=\{(\varepsilon_{\bfi}^{\{2\}},\dots,\varepsilon_{\bfi}^{\{k\}})\;|\;\bfi\in\A_{\bfd,s}\}$. Then $|\B_{\bfd,s}|$ does not depend on $\bft$ and 
$$\min\left\{s\in[0,d]\;|\; |\B_{\bfd,s}|=2^{k-1}\right\}=k-1.$$
\end{lemma}
\begin{proof}
It is easy to see that $|\B_{\bfd,s}|$ does not depend on the choice of $\bft$. By \Cref{lemma: determine epsilon}, given $a_2,\dots,a_{k}\in\{0,1\}$, then $((-1)^{a_2},\dots,(-1)^{a_{k}})\in\B_{\bfd,s}$ if and only if there exist a $k$-ple $\bfi\in\A_{\bfd,s}$ such that $t_j-i_j\equiv a_j\bmod 2$ for any $j=2,\dots,k$ (note that there is no condition on $i_1$). In particular, for the $k$-ple $(-1,\dots,-1)$ the corresponding $\bfi$ has to belong to $\A_{\bfd,s}$ with $s\geq k-1$. Thus, $\min\{s\in[0,d]\;|\;|\B_{\bfd,s}|=2^{k-1}\}\geq k-1$ and it remains to show that the equality holds. Without loss of generality, we can suppose $d_1\geq d_2\geq\dots\geq d_k\geq 3$. Let
$$c\coloneqq\min\{r\in[1,k]\;|\;d_1+\dots+d_r\geq k-1\}.$$
Since $|\B_{\bfd,s}|$ does not depend on $\bft$, we can choose
$$\bft=(d_1,\dots,d_{c-1},(k-1)-(d_1+\dots+d_{c-1}),0,\dots,0).$$
We have to prove that for any $a_2,\dots,a_k\in[0,1]$ there exists $\bfi\in{\A_{\bfd,k-1}}$ such that
$$(\varepsilon_{\bfi}^{\{2\}},\dots,\varepsilon_{\bfi}^{\{k\}})=((-1)^{a_2},\dots,(-1)^{a_k}),$$
which, by \Cref{lemma: determine epsilon}, is equivalent to show that there exists $\bfi\in{\A_{\bfd,k-1}}$ such that $t_j-i_j\equiv a_j\bmod 2$ for any $j=2,\dots,k$.
To do that, set
\begin{gather*}
E=\left\{e_1,\dots,e_u\;|\;2\leq e_j\leq c,\, a_{e_j}=0\right\}, \quad F=\left\{f_1,\dots,f_v\;|\;2\leq f_j\leq c,\, a_{f_j}=1\right\}\\
G=\left\{g_1,\dots,g_w\;|\;c+1\leq g_j\leq k,\, a_{g_j}=0\right\}, \quad H=\left\{h_1,\dots,h_z\;|\;c+1\leq h_j\leq k,\, a_{h_j}=1\right\}.
\end{gather*}
Suppose $v\leq z$ and consider the following procedure:
\begin{itemize}
    \item apply $\Delta_{f_1,h_1}\dots \Delta_{f_v,h_v}$;
    \item take $0\leq b_1\leq\lfloor t_{e_1}/2\rfloor$ and apply $\Delta_{e_1,h_{p}}$ for $p\in[v+1,v+2b_1]$;
    \item for $q=2,\dots,u$ take $0\leq b_q\leq \lfloor t_{e_q}/2\rfloor$ and apply $\Delta_{e_q,h_{p}}$ for any 
    $$p\in\left[v+1+2\sum_{l=1}^{q-1}b_l,v+2\sum_{l=1}^{q}b_l\right];$$
    \item take $0\leq c_1\leq\lfloor (t_{f_1}-1)/2\rfloor$ and apply $\Delta_{f_1,h_{p}}$ for $$p=\left[v+2\sum_{l=1}^{u}b_l+1,v+2\sum_{l=1}^{u}b_l+2c_1\right];$$
    \item for $q=2,\dots,v$ take $0\leq c_q\leq \lfloor (t_{f_q}-1)/2\rfloor$ and apply $\Delta_{f_q,h_{p}}$ for any 
    $$p\in\left[v+1+2\sum_{l=1}^ub_l+2\sum_{l=1}^{q-1}c_l,v+2\sum_{l=1}^ub_l+2\sum_{l=1}^{q}c_l\right];$$
    \item take $0\leq m\leq t_1$ such that $m+v+2\sum_{l=1}^ub_l+2\sum_{l=1}^{v}c_l=z$
    and apply $\Delta_1,h_{p}$ for any 
    $$p\in\left[v+2\sum_{l=1}^ub_l+2\sum_{l=1}^{v}c_l+1,z\right].$$
\end{itemize}
Let $\bfi$ the $k$-ple obtained by applying this procedure to $\bft$. Then, $t_j-i_j\equiv a_j\bmod 2$ for any $j=2,\dots,k$ and, by construction, $\bfi\in\A_{\bfd,k-1}$. Thus, it is enough to prove that there exist $b_1,\dots,b_u,c_1,\dots,c_v,m$ such that $m+v+2\sum_{l=1}^ub_l+2\sum_{l=1}^{v}c_l=z$, and this is equivalent to show that the maximum value of the lhs of the equation is at least the maximum value of the rhs of the equation. We have $z\leq k-c$ and
\begin{gather*}
t_1+v+2\sum_{j\in E}^u\left\lfloor\frac{t_j}{2}\right\rfloor+2\sum_{j\in F}^{v}\left\lfloor\frac{t_j-1}{2}\right\rfloor\geq t_1+v+2\sum_{j\in E}\left(\frac{t_j}{2}-\frac{1}{2}\right)+2\sum_{j\in F}\left(\frac{t_j}{2}-1\right)=\sum_{j=1}^ct_j-u-v=k-c,
\end{gather*}
and this ends the proof in the case $v\leq z$. For $v>z$ it is enough to consider a procedure analogue to the one described.
\end{proof}
By \Cref{remark: W-state sta nel quadrello}, for any $J\subseteq[2,k]$, there exists $T^J\in\spann(\Cc_{\bfd}^J)$ such that $W_{d_1}\otimes\dots\otimes W_{d_k}\in \mathrm{span}\{\cup_{J\subseteq [2,k]}T^J\}$. For convenience we work with a multiple of $W_{d_1}\otimes \cdots \otimes W_{d_k}$, in order to simplify computations. To find the $T^J$'s we have to solve the linear system 
\begin{align}\label{eq: famoso sistema}
\mathcal S: \prod_{j=1}^kd_jW_{d_1}\otimes\dots\otimes W_{d_k}= \sum_{J \subseteq[2,k]}  T^{J},
\end{align}
whose unknowns are the $\alpha_s^J$'s with $J\subseteq[2,k]$ and $0\leq s\leq d$, and whose equations are labelled by $\A_{\bfd,s}$. We order $2^{[2,k]}$ as follows: given $J_1,J_2\subseteq[2,k]$ with $J_1\neq J_2$, then $J_1>J_2$ if and only if $|J_1|<|J_2|$ or $|J_1|=|J_2|$ and $\min J_1<\min J_2$.  We order the unknowns as follows: given $\alpha_{s_1}^{J_1}$ $\alpha_{s_2}^{J_2}$, we have $\alpha_{s_1}^{J_1}>\alpha_{s_2}^{J_2}$ if $s_1<s_2$ or $s_1=s_2$ and
$J_1>J_2$, where the subsets of $[2,k]$ are ordered by lexicographic order. 

Considering the structure of the $T^J$'s, see \Cref{remark: free parameters rnc}, we note that $\mathcal S$ can be split in $d+1$ independent linear systems $\mathcal S_s$, with $s=0,\dots,d$, where the unknowns of $\mathcal S_s$ are $(\alpha_s^J)_{J\subseteq[2,k]}$. Moreover, $\mathcal S_s$ is homogeneous for any $s\neq d-k$. In the next proposition, we exhibit a non-trivial solution for any $\mathcal S_s$ admitting one.

\begin{proposition}\label{prop: mortale}
Let $k\in\NN$ and $\bfd\in\NN^k$ with $k\geq 2$ and $d_j\geq 3$. For any $s=0,\dots,d$ choose $\bft^{(s)}\in\A_{\bfd,s}$ with $\bft^{(d-k)}=(d_1-1,\dots,d_k-1)$, and set $\varepsilon_{\bft^{(s)}}^J=1$ for any $J\subseteq[2,k]$. The subsystem $\mathcal S_s$ of the system \eqref{eq: famoso sistema} admits a non-zero solution if and only if $s\in[0,k-2]\cup\{d-k\}\cup[d-k+2,d]$. For $s\in[0,k-2]\cup[d-k+2,d]$ a non-zero solution of the subsystem $\mathcal{S}_s$ is 
$$\left(\alpha_s^J\right)_{J\subseteq[2,k]}=\left(-1^{|J|}\right)_{J\subseteq[2,k]}.$$
The unique solution of $\mathcal S_{d-k}$ is given by $\alpha_{d-k}^J=2^{1-k}$ for any $J\subseteq[2,k]$.
  
\end{proposition}
\begin{proof}
Let $s\in[0,k-2]$. The equations of $\mathcal S_s$ are
$$\mathcal S_{s,\bfi}:\sum_{J\subseteq[2,k]}\varepsilon_{\bfi}^J\alpha_s^J=0$$
as $\bfi$ varies in $\A_{\bfd,s}$. For any $\bfi\in\A_{\bfd,s}$, by \Cref{lemma: determine epsilon}, we get $\varepsilon_{\bfi}^J=(-1)^{|N_{\bfi}\cap J|}$, where we recall that 
$N_{\bfi}=\{i_j\;|\;t^{(s)}_j-i_j\equiv1\bmod2,\, j\geq 2\}.$
By substituting $\alpha_s^J=(-1)^{|J|}$ in $\mathcal S_{s,\bfi}$, we have
$$\mathcal S_{s,\bfi}:\sum_{J\subseteq[2,k]}(-1)^{|N_\bfi\cap J|+|J|}=0,$$
which is verified by \Cref{lemma: even and odd}. By symmetry, the same argument works for all $s\in [d-k+2,d]$.

Now we want to better understand the matrix $M_s$ associated with the subsystem $\mathcal S_s$. The matrix $M_s$ has size $|\mathcal{A}_{\bfd ,s}|\times 2^{k-1}$, the first $k-1$ columns of $M_s$ are $( \varepsilon^{\{r\}}_{\bfi})_{\bfi}$ for $r=2,\dots,k$ and moreover all other columns are the Hadamard product of some of these. More precisely, the column labelled by $J$ is $(\varepsilon_{\bfi}^J)_{\bfi\in\A_{\bfd,s}}$ and thus, since 
$$\varepsilon_{\bfi}^J=\prod_{r\in J}\varepsilon_{\bfi}^{\{r\}},$$
the column $(\varepsilon_{\bfi}^J)_{\bfi\in\A_{\bfd,s}}$ is the Hadamard product of the columns $(\varepsilon_{\bfi}^{\{r\}})_{\bfi\in\A_{\bfd,s}}$ with $r\in J$. By \Cref{lemma: hadamard if greater than k-2}, the minimum $s$ such that the matrix $M_s$ has a submatrix that is a Hadamard matrix of size $2^{k-1}$ is $s=k-1$. Moreover, for any 
$k-1 \leq s\leq d-k+1$ the matrix $M_s$ still has a submatrix that is a Hadamard matrix of size $2^{k-1}$. Indeed, this easily follows from the fact that for all $s \in [k-1,d-k+1]$ there is a subset of $\A_{\bfd,s}$ in bijection with $\mathcal{A}_{\bfd,k-1}$. It follows that, since Hadamard matrices are full rank, the system $\mathcal S_s$ has a unique solution if $s\in [k-1,d-k+1]$. In particular, since $\mathcal S_s$ is homogeneous for all $s\in [k-1,d-k+1]\setminus \{ d-k \}$, then the unique solution of $\mathcal S_s$ is the zero solution for any $s\in [k-1,d-k+1]\setminus \{ d-k \}$. It remains to solve $\mathcal S_{d-k}$. Recall that
$$z_{\bfi}\left(\prod_{i=1}^kd_iW_{d_1}\otimes\dots\otimes W_{d_k}\right)=
\begin{cases}
1 & \text{if $\bfi=\bft^{(d-k)}$},\\
0 & \text{otherwise}
\end{cases},
$$
and $\varepsilon_{\bft^{(d-k)}}^J=1$ for any $J\subseteq[2,k]$. Thus 
$$\mathcal S_{d-k},\bft^{(d-k)}:\sum_{J\subseteq[2,k]}\alpha_{d-k}^J=1$$
and, for any $\bfi\neq \bft^{(d-k)}$, we have
$$\mathcal S_{d-k},\bfi:\sum_{J\subseteq[2,k]}\varepsilon_{\bfi}^J\alpha_{d-k}^J=0.$$
For $\bfi\neq\bft^{(d-k)}$ we have $N_\bfi\neq\emptyset$, thus, using lemma \Cref{lemma: determine epsilon} and \Cref{lemma: even and odd}, we have
$$\sum_{J\subseteq[2,k]}(\varepsilon_{\bfi}^J2^{1-k})=2^{1-k}\sum_{J\subseteq[2,k]}(-1)^{|N_{\bfi}\cap J|}=0.$$
Finally, noting that
$$\sum_{J\subseteq[2,k]}2^{1-k}=2^{k-1}2^{1-k}=1,$$
we conclude that the unique solution of $\mathcal S_{d-k}$ is given by $\alpha_{d-k}^J=2^{1-k}$ for any $J\subseteq[2,k]$.
\end{proof}

Before proceeding further, let us clarify with our running example the results of \Cref{prop: mortale} by explicitly solving the system \eqref{eq: famoso sistema} and understanding how it actually splits in independent subsystems.

\begin{example}\label{example: step3}
Let $k=2$ and $\bfd=(3,3)$. We write down and solve the linear system 
$$\mathcal S:9W_3\otimes W_3=T+T^{\{2\}},$$
where $T\in\spann(\mathcal C_{(3,3)})$ and $T^{\{2\}}\in\spann\mathcal C_{(3,3)}^{\{2\}}$.
By \Cref{example: step2}, we have
$$T=(\alpha_0,\alpha_1,\alpha_1,\alpha_2,\alpha_2,\alpha_2,\alpha_3,\alpha_3,\alpha_3,\alpha_3,\alpha_4,\alpha_4,\alpha_4,\alpha_5,\alpha_5,\alpha_6),$$
$$T^{\{2\}}=(\alpha_0^{\{2\}},\alpha_1^{\{2\}},-\alpha_1^{\{2\}},\alpha_2^{\{2\}},-\alpha_2^{\{2\}},\alpha_2^{\{2\}},\alpha_3^{\{2\}},-\alpha_3^{\{2\}},\alpha_3^{\{2\}},-\alpha_3^{\{2\}},-\alpha_4^{\{2\}},\alpha_4^{\{2\}},-\alpha_4^{\{2\}},\alpha_5^{\{2\}},-\alpha_5^{\{2\}},\alpha_6^{\{2\}}),$$
and, since the only non-zero coordinate of $W_3\otimes W_3$ is $z_{2,2}$, the matrix of the linear system is
\[
\setlength{\arraycolsep}{2pt}
\mathcal S:\begin{pmatrix}
1 & 1 & & & & \\
& & \begin{bmatrix}1 & 1 \\ 1 & -1\end{bmatrix} & & & & \\
& & & \ddots & & & \\
& & & &\begin{bmatrix}1 & 1 \\ 1 & -1\end{bmatrix} & & \\
& & & & & 1 & 1
\end{pmatrix} \begin{pNiceArray}{c}
\alpha_0\\
\alpha_0^{\{2\}}\\
\vdots\\
\vdots \\
\alpha_6\\
\alpha_6^{\{2\}}
\end{pNiceArray}=\begin{pmatrix}
 0\\
 0\\
 \vdots\\
 \vdots\\
 0\\
 0
\end{pmatrix}
\]where we already removed repeated equations. The systems $\mathcal S$ splits into the independent systems
\begin{gather*}\mathcal S_0:\alpha_0+\alpha_0^{\{2\}}=0,\quad \mathcal S_1:\begin{cases}
\alpha_1+\alpha_1^{\{2\}}=0\\
    \alpha_1-\alpha_1^{\{2\}}=0\\    
\end{cases},\quad \mathcal S_2:\begin{cases}
\alpha_2+\alpha_2^{\{2\}}=0\\
    \alpha_2-\alpha_2^{\{2\}}=0\\
\end{cases},\quad \mathcal S_3:\begin{cases}
\alpha_3+\alpha_3^{\{2\}}=0\\
    \alpha_3-\alpha_3^{\{2\}}=0\\    
\end{cases},\\
\mathcal S_4:\begin{cases}
\alpha_4+\alpha_4^{\{2\}}=1\\
    \alpha_4-\alpha_4^{\{2\}}=0\\
\end{cases}, \quad \mathcal S_5:\begin{cases}
\alpha_5+\alpha_5^{\{2\}}=0\\
    \alpha_5-\alpha_5^{\{2\}}=0\\
\end{cases}, \quad \mathcal S_6:\alpha_6+\alpha_6^{\{2\}}=0.
\end{gather*}
As predicted by \Cref{prop: mortale}, $\mathcal S_s$ admits a non-trivial solution if and only if $s=0,4,6$ and the unique solution of $\mathcal S_4$ is $(\alpha_4,\alpha_4^{\{2\}})=(2^{-1},2^{-1})$. Moreover, the matrix of $\mathcal S_s$ is a Hadamard matrix for any $s=1,\dots,5$, as we expected by \Cref{lemma: hadamard if greater than k-2} and the proof of \Cref{prop: mortale}. More precisely, any solution of $\mathcal S$ is of the form
$$
\begin{array}{c} {\scriptstyle \alpha_0} \\ (\alpha_0 ,\end{array}
\begin{array}{c} {\scriptstyle \alpha_0^{\{2\}}} \\ -\alpha_0, \end{array}
\begin{array}{c} {\scriptstyle \alpha_1} \\ 0, \end{array}
\begin{array}{c} {\scriptstyle \alpha_1^{\{2\}}} \\ 0 ,\end{array}
\begin{array}{c} {\scriptstyle \alpha_2} \\ 0, \end{array}
\begin{array}{c} {\scriptstyle \alpha_2^{\{2\}}} \\ 0, \end{array}
\begin{array}{c} {\scriptstyle \alpha_3} \\ 0, \end{array}
\begin{array}{c} {\scriptstyle \alpha_3^{\{2\}}} \\ 0 ,\end{array}
\begin{array}{c} {\scriptstyle \alpha_4} \\ 2^{-1} ,\end{array}
\begin{array}{c} {\scriptstyle \alpha_4^{\{2\}}}  \\ 2^{-1}, \end{array}
\begin{array}{c} {\scriptstyle \alpha_5} \, \\ 0, \end{array}
\begin{array}{c} {\scriptstyle \alpha_5^{\{2\}}} \\ 0,\end{array}
\begin{array}{c} {\scriptstyle \alpha_6} \\ \alpha_6 ,\end{array}
\begin{array}{c} {\scriptstyle \alpha_6^{\{2\}}} \\ -\alpha_6) \end{array}.
$$

Substituting in the expression of the general elements $T$ and $T^{\{2\}}$, we have that $W_3\otimes W_3=T+T^{\{2\}}$ with $T\in\spann(\mathcal C_{(3,3)})$ and $T^{\{2\}}\in\spann\mathcal C_{(3,3)}^{\{2\}}$ given by
\begin{gather*}
T=\alpha_0x_{1,1}^3\otimes x_{2,1}^3+2^{-1}(3x_{1,0}^3\otimes x_{2,0}x_{2,1}^2+9x_{1,0}^2x_{1,1}\otimes x_{2,0}^2x_{2,1}+3x_{1,0}x_{1,1}^2\otimes x_{2,0}^3)+\alpha_6x_{1,0}^3\otimes x_{2,0}^3,\\
T^{\{2\}}=-\alpha_0x_{1,1}^3\otimes x_{2,1}^3+2^{-1}(-3x_{1,0}^3\otimes x_{2,0}x_{2,1}^2+9x_{1,0}^2x_{1,1}\otimes x_{2,0}^2x_{2,1}-3x_{1,0}x_{1,1}^2\otimes x_{2,0}^3)-\alpha_6x_{1,0}^3\otimes x_{2,0}^3.
\end{gather*}

\end{example}

\section{Computing the bounds}\label{section: computing the bounds}

This section is devoted to prove our main \Cref{theorem: bound rango}. Before proceeding, we briefly recall some elements of the theory of catalecticants that will be needed later; for a more detailed account, we refer for instance to \cite[Chapter 1]{IarrobinoKanev} or to the survey \cite{guida}. We also briefly recall the classical Sylvester's algorithm \cite{sylvesterAlgo} used to compute the symmetric rank decomposition of a homogeneous polynomial and we refer to \cite{comas2011Sylv, sylvMod} for a detailed modern treatment.

\subsection*{Catalecticants and Sylvester's algorithm}
Let $F\in \CC[u,v]_{e}$ be a degree $e$ homogeneous polynomial:
$$
F=\sum_{i=0}^e a_i \binom{e}{i}u^{e-i}v^i.
$$
Let us briefly explain the Sylvester's algorithm to compute the symmetric rank of $F$, i.e., the minimum integer $\rk_s(F)$ such that
$F=\sum_{i=1}^{\rk_s(F)}L_i^d$ for linear forms $L_i$, as well as to get a length $\rk_s(F)$ decomposition of $F$. We also remark that for a binary form $F$, we have that $\rk_s(F)=\rk(F)$ \cite{zhang2016comon}, so from now on we will just talk about the rank of a binary form. 

Let $\CC[\partial_u,\partial_v]_m$ be the space of homogeneous polynomials of degree $m$ in the variables $\partial_u,\partial_v$. For $m\leq e$, the \emph{m-catalecticant map} of $F$ is the linear map $ \Cat_{m}(F):\CC[\partial_u,\partial_v]_{m}\rightarrow \CC[u,v]_{e-m}$ defined by derivation. If we fix basis $(\partial_u^{i}\partial_v^{m-i})_{i=0,\dots,m}$ on the domain and $\left(\binom{e-m}{i}u^{e-m-i}v^i\right)_{i=0,\dots,m}$ on the codomain, and we denote 
$$
\Cat_{m}(F)=\begin{bmatrix}
    a_0 & \cdots & a_{m}\\
    \vdots & & \vdots \\
    a_{e-m}& \dots & a_{e}
\end{bmatrix},
$$
then the matrix associated to the catalecticant map is $e\cdot \Cat_m(F)$.
With a slight abuse of notation, from now on we denote by $\Cat_m(F)$ both the linear operator and the corresponding matrix where we fixed basis as just explained. Clearly $\Cat_m(F)=(\Cat_{e-m}(F))^t$ and moreover $\rk(\Cat_m(F))\leq \rk(\Cat_{m+1}(F))$ for all $m\leq e/2$. Catalecticant matrices give also a bound for the rank of $F$ since $\rk(\Cat_{m}(F))\leq \rk(F)$ for all $m$. Hence, it is natural to directly look at the so-called \emph{most square} catalecticant matrix $\Cat_{m}(F)$, where $m=\lfloor e/2\rfloor$.

Let $F\in \CC[u,v]_e$ and let $r=\rk(\Cat_{\lfloor e/2\rfloor}(F))$. Then, by \cite{comas2011Sylv}, we have
\begin{align}\label{eq: algo sylvester}
   \rk(F)=\begin{cases}
    r, & \text{ if } \ker(\Cat_{r}(F)) \text{ contains a square-free element}\\
e-r+2, & \text{otherwise}.
\end{cases}
\end{align}
 In the first case, a decomposition of $F$ is found by looking at the distinct roots $(\alpha_i,\beta_i)$ of a square-free element of $\ker(\Cat_{r}(F))$ seen as a degree-$r$ polynomial. To get a decomposition of length $\rk(F)$, one has to solve the linear system $F=\lambda_1 (\alpha_1u+\beta_1v)^e+\cdots +\lambda_r(\alpha_r u+\beta_rv)^e$ and find all $\lambda_i$'s.

In the second case, a similar procedure is applied by looking at a square-free element of $\ker(\Cat_{d-e+2}(F))$, computing its distinct roots and solving the corresponding linear system to find the coefficient of the forms providing a rank-decomposition. 

It is known that, if $r$ is the rank of the most square catalecticant matrices then, $\dim(\ker(\Cat_r))=2$ if $2r-2=e$ and $\dim(\ker(\Cat_r))=1$ otherwise.

\subsection*{Computing the bounds}
Now we minimize the rank of a family of binary forms that will be used to compute the bound of  \Cref{theorem: bound rango}.  
\begin{proposition}\label{prop: rango TJ}
Let $k\geq 2$, let $\bfd\in \NN_{\geq 3}^k$ and let $d=d_1+\cdots +d_k$. Let $\mathcal{F}\subseteq \CC[u,v]_d$ be the family of degree $d$-homogenous polynomial defined as 
$$
\mathcal{F}= \left\{  F=\sum_{i=0}^{k-2}a_i\binom{d}{i}u^{d-i}v^i+\binom{d}{d-k}a_{d-k}u^{k}v^{d-k}+\sum_{i=d-k+2}^d a_i \binom{d}{i}u^{d-i}v^i, \text{ for }a_i\in \CC \text{ with } a_{d-k}\neq 0 \right\}.
$$
Then 
$$
\min_{F\in\mathcal F}\rk F=d-2k+2.
$$
In particular, if  $2k\leq\left\lfloor \frac d2\right\rfloor +1$ the minimum is achieved by any $F\in\mathcal F$ with $a_{k-2}\neq 0$. Otherwise, the minimum is achieved by any $F\in \mathcal{F}$ with $a_{k-2}\neq 0$ and $a_i=0$ for $i\in \{0,\dots,k-3\} \cup \{d-k+2,\dots,d \}$.
\end{proposition}
\begin{proof}
Assume $2k\leq\left\lfloor \frac d2\right\rfloor +1$. The most square catalecticant matrix associated to $F$ is 
$$
\Cat_{\left\lfloor\frac{d}{2}\right\rfloor}(F)=\begin{pmatrix}
M & 0 & 0\\
0 & 0 & 0\\
0 & 0 & N
\end{pmatrix}
$$
where
\begin{align}\label{eq: M ed N}
M=\begin{bmatrix}
a_0 &  \dots & a_{k-3}&  a_{k-2} \\
a_1 &  & \iddots & 0\\
\vdots & \iddots & \iddots& \vdots\\
a_{k-2}& 0& \dots & 0
\end{bmatrix}\in\CC^{(k-1)\times(k-1)}, \quad 
N=\begin{bmatrix}
   0& \dots & \dots & 0& a_{d-k}\\
   \vdots& & &\iddots &0\\
   \vdots& & \iddots&\iddots & a_{d-k+2}\\
   0 & \iddots& \iddots &\iddots & \vdots \\
   a_{d-k}& 0& a_{d-k+2} & \dots & a_{d}
\end{bmatrix}\in\CC^{(k+1)\times(k+1)}.\end{align}
Since $k-1+k+1=2k\leq\left\lfloor\frac d2\right\rfloor+1$, $M$ and $N$ do not share any row or any column. Moreover, since $\rk N=k+1$, it follows that $\rk \Cat_{\left\lfloor\frac d2\right\rfloor}(F)=k+1+\rk M$. By Sylvester algorithm (\cref{eq: algo sylvester}), we get that $\rk F=k+1+\rk M$ if $\ker(\Cat_{k+1+\rk M}(F))$ contains a square-free element, and $\rk F=d-(k+1+\rk M)+2$ otherwise.

First, notice that $\rk M \in [0,k-1]$ and $\rk M=i$ if and only if $a_{i-1}\neq 0$ and $a_{i}=\cdots =a_{k-2}=0$; in particular, if $\rk M=i$, then the smallest non-zero sub-matrix of $M$ is made of the first $i$ columns and $i$ rows.
For all $i=0,\dots,k-1$ both $M$ and $N$ are a blocks of the catalecticant matrix $\Cat_{k+1+i}(F)$ and the assumption $2k\leq\left\lfloor \frac d2\right\rfloor +1$ guarantees that $N$ and the non-zero sub-matrix of $M$ do not share any row or column, unless $k=\frac{d+2}{4}$. 

First, assume $k\neq\frac{d+2}{4}$. The shape of $\Cat_{k+1+i}(F)$ implies that $\ker(\Cat_{k+1+i}(F)) =\langle\partial_u^{k+1}\partial_v^i\rangle$, and thus $\rk F=d-(k+1+\rk M)+2$. Hence, in this case $\displaystyle\min_{F\in\mathcal F}\rk F$ is obtained for $\rk M=k-1$ and $\displaystyle\min_{F\in\mathcal F}\rk F=d-2k+2$. 

Assume now $k=\frac{d+2}{4}$. Repeating the same argument, it is easy to see that for any $i=0,\dots,k-2$ we have that $\rk F=d-(k+1+\rk M)+2$ and for $i=k-1$ we have $\rk F=\frac{d+2}{2}$. Hence, also in this case $\displaystyle\min_{F\in\mathcal F}\rk F=d-2k+2$.

\noindent Assume now $2k>\left\lfloor \frac d2\right\rfloor+1$ and set
$$h_1=2k-\left\lfloor \frac d2\right\rfloor-1, \quad h_2=2k-\left\lceil \frac d2\right\rceil-1.$$
In this case the matrix associated with the most square catalecticant map $\Cat_{\left\lfloor\frac{d}{2}\right\rfloor}(F)$ has the form
\NiceMatrixOptions{cell-space-limits = 2pt}
\begin{equation*}\Cat_{\left\lfloor\frac{d}{2}\right\rfloor}(F)=\begin{pNiceArray}{cc|c|cc}
M_{1,1} & M_{1,2} & M_2 & 0 & 0\\
M_{1,3} & M_{1,4} & 0 & 0 & 0\\ \hline 
M_3 & 0 & 0 & 0 & N_3\\ \hline
0 & 0 & 0 & N_{1,4} & N_{1,2}\\
0 & 0 & N_2 & N_{1,3} & N_{1,1}
\end{pNiceArray},
\end{equation*}
where
\begin{gather*}
M_{1,4}\in\CC^{(d-3k+1)\times(d-3k+1)},\quad M_2\in\CC^{h_1\times h_1},\quad M_3\in\CC^{h_2\times h_2},\\
N_{1,4}\in\CC^{(d-3k+3)\times(d-3k+3)},\quad N_2\in \CC^{h_1\times h_1},\quad N_3\in\CC^{h_2\times h_2},
\end{gather*}
and, defining $M$ and $N$ as in \cref{eq: M ed N}, we have 
\begin{equation*}
\begin{pNiceArray}{cc|c}
M_{1,1} & M_{1,2} & M_2 \\
M_{1,3} & M_{1,4} & 0 \\ \hline
M_3 & 0 & 0
\end{pNiceArray}=M, \hbox{ and }\begin{pNiceArray}{c|cc}
 0  & 0 & N_3\\ \hline
 0 & N_{1,4} & N_{1,2}\\
 N_2 & N_{1,3} & N_{1,1}
\end{pNiceArray}=N.
\end{equation*}
As before, we want to bound the $\displaystyle \min_{F \in \mathcal F}\rk F$ and so we specialize the form $F$ by imposing the following conditions on its coefficients: $$
\begin{cases}
 a_0=a_1=\dots=a_{k-3}=0, \\
 a_{d-k+2}=a_{d-k+3}=\dots=a_{d}=0.
\end{cases}
$$
Using the fact that, by hypothesis $a_{d-k}\neq 0$, the previous conditions on the coefficients gives $$\rk \Cat_{\left\lfloor\frac{d}{2}\right\rfloor}(F)=\left\lfloor\frac{d}{2}\right\rfloor+1-h_2$$ for any choice of $a_{k-2} \neq 0$. 
Indeed, the first $h_2$ columns and the last $h_2$ columns of $\Cat_{\left\lfloor\frac{d}{2}\right\rfloor}(F)$ are proportional to each other and removing the last $h_2$ columns of $\Cat_{\left\lfloor\frac{d}{2}\right\rfloor}(F)$ we get
\begin{equation*}\begin{pNiceArray}{cc|c|c}
0 & 0 & M_2 & 0 \\
0 & M_{1,4} & 0 & 0 \\ \hline 
M_3 & 0 & 0 & 0 \\ \hline
0 & 0 & 0 & N_{1,4} \\
0 & 0 & N_2 & 0
\end{pNiceArray}
\end{equation*}
with $M_{1,4},\,M_2,\,M_3,\,N_{1,4},\,N_2$ anti-diagonal matrices with non-zero anti-diagonal entries $a_{k-2}$ for $M_{1,4},\,M_2,$ $M_3$ and $a_{d-k}$ for $N_{1,4},\,N_2$.

Now, following Sylvester's algorithm we have to look if  $\ker(\Cat_{\left\lfloor\frac{d}{2}\right\rfloor+1-h_2}(F))$ contains a square-free binary form of degree $\left\lfloor\frac{d}{2}\right\rfloor+1-h_2$. Note that $\Cat_{\left\lfloor\frac{d}{2}\right\rfloor+1-h_2}(F)$ can be obtained by $\Cat_{\left\lfloor\frac{d}{2}\right\rfloor}(F)$ by cut and paste the last $h_2-1$ columns of $\Cat_{\left\lfloor\frac{d}{2}\right\rfloor}(F)$ as last $h_2-1$ rows of $\Cat_{\left\lfloor\frac{d}{2}\right\rfloor+1-h_2}(F)$, with the first $h_2-1$ bottom-left entries removed. By the previous description, the columns $A_1$ and $A_{\left\lfloor\frac{d}{2}\right\rfloor+2-h_2}$ of $\Cat_{\left\lfloor\frac{d}{2}\right\rfloor+1-h_2}(F)$ are proportional, in particular $a_{k-2}A_{\left\lfloor\frac{d}{2}\right\rfloor+1}=a_{d-k}A_1$. Hence, since 
$$\dim(\ker(\Cat_{\left\lfloor\frac{d}{2}\right\rfloor+1-h_2}(F)))=1,$$
we get
$$\ker(\Cat_{\left\lfloor\frac{d}{2}\right\rfloor+1-h_2}(F))=\langle a_{d-k}\partial ^{{\left\lfloor\frac{d}{2}\right\rfloor+1-h_2}}_u-a_{k-2}\partial^{\left\lfloor\frac{d}{2}\right\rfloor+1-h_2}_v\rangle.$$
It follows that, for any non-zero choice of $a_{k-2} \neq 0$, there is always a square-free binary form in  $\ker(\Cat_{\left\lfloor\frac{d}{2}\right\rfloor+1-h_2}(F))$.
This shows that
$$\min_{F \in \mathcal F}\rk F \leq \left\lfloor\frac{d}{2}\right\rfloor+1-h_2=d-2k+2.$$ To prove that the equality holds, we have to show that, for any $F\in\mathcal F$ and any $k+1<r<d-2k+2$, if $\rk\Cat_{\left\lfloor\frac{d}{2}\right\rfloor}(F)=r$ then $\Cat_r(F)$ has no square-free element in the kernel. Since $r<d-2k+2$, in order to have $\rk\Cat_{\left\lfloor\frac{d}{2}\right\rfloor}(F)=r$ we have to impose that all the $(d-2k+2)\times(d-2k+2)$ minors of $\Cat_{\left\lfloor\frac{d}{2}\right\rfloor}(F)$ are zero. In particular, starting from the minor
\begin{equation*}\begin{pNiceArray}{c|c|cc}
M_{1,4} & 0 & 0 & 0\\ \hline 
0 & 0 & 0 & N_3\\ \hline
0 & 0 & N_{1,4} & N_{1,2}\\
0 & N_2 & N_{1,3} & N_{1,1}
\end{pNiceArray}
\end{equation*}
and noting that $M_{1,4}$ is anti-upper-triangular with all the anti-diagonal elements equal to $a_{k-2}$, we get $a_{k-2}=0$. Considering the minor obtained by shifting the first $d-3k+1$ columns of the previous one to the left by one in $\Cat_{\left\lfloor\frac{d}{2}\right\rfloor}(F)$, we obtain $a_{k-3}=0$. By repeating this procedure, and shifting the first $d-3k+1$ rows of the minor upwards in $\Cat_{\left\lfloor\frac{d}{2}\right\rfloor}(F)$ when necessary, we find $a_0=\dots=a_{k-2}=0$. Thus, if $\rk F<d-2k+2$ then 
\begin{equation*}\Cat_{\left\lfloor\frac{d}{2}\right\rfloor}(F)=\begin{pNiceArray}{cc|c|cc}
0 & 0 & 0 & 0 & 0\\
0 & 0 & 0 & 0 & 0\\ \hline 
0& 0 & 0 & 0 & N_3\\ \hline
0 & 0 & 0 & N_{1,4} & N_{1,2}\\
0 & 0 & N_2 & N_{1,3} & N_{1,1}
\end{pNiceArray}
\end{equation*}
and, in particular, $\rk \Cat_{\left\lfloor\frac{d}{2}\right\rfloor}(F)=k+1$ (recall that $N$ has always rank equal to $k+1$). The matrix $N$ is a sub-matrix of $\Cat_{k+1}(F)$, i.e.
\begin{equation*}
\Cat_{k+1}=\begin{pNiceArray}{c|c}
0 & 0 \\ \hline
0 & N
\end{pNiceArray}
\end{equation*}
and, since $N$ is full-rank, we get 
$\ker\Cat_{k+1}=\langle\partial_u^{k+1}\rangle,$ and this concludes the proof. 
\end{proof}

\begin{remark}\label{remark: multi hom binary in forma binaria emplice}
Let $J\subseteq [2,k]$,  $\bfd\in\NN_{\geq 3}^k$, and denote by  $V_{\bfd}^J\subset\sym^{d_1}\CC^2\otimes\cdots\otimes\sym^{d_k}\CC^2$
the vector subspace such that $\PP(V_\bfd^J)=\spann(\mathcal C_\bfd^J).$
By \Cref{remark: free parameters rnc}, a basis of $V_\bfd^J$ is given by
$$\left(\sum_{\bfi\in\mathcal A_{\bfd,s}}\varepsilon_\bfi^J\prod_{j=1}^k\binom{d_j}{i_j}x_{1,0}^{i_1}x_{1,1}^{d_1-i_1}\otimes\cdots\otimes x_{k,0}^{i_k}x_{k,1}^{d_k-i_k}\right)_{s=0,\dots,d},$$
for suitable $\varepsilon_\bfi^J\in\{-1,1\}.$ Since $\mathcal C_\bfd^J$ is a rational normal curve, $V_\bfd^J$ is isomorphic to the vector space of binary forms of degree $d$, and an isomorphism is given by
$$\begin{array}{cccc}
   \varphi^J: &  V_\bfd^J & \longrightarrow &\CC[u,v]_{d}\\
   &  \\
   & \sum_{\bfi\in\mathcal A_{\bfd,s}}\varepsilon_\bfi^J\prod_{j=1}^k\binom{d_j}{i_j}x_{1,0}^{i_1}x_{1,1}^{d_1-i_1}\otimes\cdots\otimes x_{k,0}^{i_k}x_{k,1}^{d_k-i_k}
 &\mapsto
    & \binom{d}{i_1+\dots+i_k}u^{d-(i_1+\cdots +i_k)}v^{i_1+\cdots +i_k}
\end{array}.$$
The isomorphism $\varphi^J$ sends the points of $\mathcal C_\bfd^J$ in powers of linear forms. Hence, the $\mathcal C_\bfd^J$-rank of any element $T^J\in\spann(\mathcal C_\bfd^J)$ is equal to the (symmetric) rank of $\varphi^J(T^J)$.
\end{remark}

We have now all the tools to prove our main \Cref{theorem: bound rango}, which we restate below for convenience.
\begin{theorem*}(\Cref{theorem: bound rango}) 
Let $k\geq 2$, $\bfd\in \NN_{\geq 3}^k$ and $d=d_1+\cdots +d_k$.  Then
$$
\rk_{\bfd}(W_{d_1}\otimes \cdots \otimes W_{d_k})\leq  2^{k-1}(d-2k+2).
$$
\end{theorem*}
\begin{proof}
By \Cref{remark: W-state sta nel quadrello} $W_{d_1}\otimes\cdots   \otimes W_{d_k}$ belongs to $\spann\left\{  \cup_{J\subseteq [2,k]}  \mathcal{C}_{\bfd}^J\right\} $, so we have
$$
\prod_{j=1}^kd_j W_{d_1}\otimes\cdots   \otimes W_{d_k}=\sum_{J\subseteq [2,k]}T^J, \,\hbox{ for some }T^J\in\spann\mathcal{C}_{\bfd}^J. 
$$
Writing $T^J$ with respect to the basis of \Cref{remark: multi hom binary in forma binaria emplice}, we have
$$T^J=\sum_{s=0}^{d}\alpha_s^J\left(\sum_{\bfi\in\mathcal A_{\bfd,s}}\varepsilon_\bfi^Jx_{1,0}^{i_1}x_{1,1}^{d_1-i_1}\otimes\cdots\otimes x_{k,0}^{i_k}x_{k,1}^{d_k-i_k}\right),$$
where the $\alpha^J_s$'s are a solution of the linear system $\mathcal S$ of \Cref{prop: mortale}.
Using the isomorphism $\varphi^J$ defined in \Cref{remark: multi hom binary in forma binaria emplice}, each $T^J$ can be seen as a binary form of degree $d$ in $\spann\mathcal C_\bfd^J$. More precisely, we have
$$\varphi^J(T^J)=\sum_{s=0}^d\alpha_s^J\binom{d}{s}u^{d-s}v^s.$$
It follows that $\rk_\bfd T^J\leq\rk\varphi^J(T^J)$, where $\rk\varphi^J(T^J)$ is the Waring rank of the binary form $\varphi^J(T^J)$ and, by \Cref{remark: multi hom binary in forma binaria emplice}, it is equal to the $\mathcal C_\bfd^J$-rank of $T^J$.
By \Cref{prop: mortale}, we have $\alpha_s^J=0$ for any $s\in[k-1,d-k-1]\cup\left\{d-k+1\right\}$ and $\alpha^J_{d-k}\neq 0$. Moreover, whether $2k\leq\left\lfloor\frac{d}{2}\right\rfloor+1$ or $2k>\left\lfloor\frac{d}{2}\right\rfloor+1$, by \Cref{prop: mortale} we can choose the $\alpha_s^J$'s such that $\varphi^J(T^J)$ is a minimizer as in \Cref{prop: rango TJ}, so that $\rk (\varphi^J(T^J))=d-2k+2$. Hence, we have
\begin{equation*}
\rk_{\bfd}(W_{d_1}\otimes \cdots \otimes W_{d_k})\leq\sum_{J\subseteq[2,k]}\rk_\bfd T^J\leq \sum_{J\subseteq[2,k]}\rk \varphi^J(T^J)= 2^{k-1}(d-2k+2).\qedhere
\end{equation*}
\end{proof}
Since for a tensor $T\in \sym^{d_1}\CC^2\otimes \cdots \otimes \sym^{d_k}\CC^2$ we have $\rk(T)\leq \rk_{\bf d}(T)$, the previous result gives a bound also on the tensor rank of $W_{d_1}\otimes \cdots \otimes W_{d_k}$.

\begin{corollary*}(\Cref{corollary: bound rango vero})
Let $k\geq 2$, let $\bfd\in \NN_{\geq 3}^k$ and let $d=d_1+\cdots+d_k$.  Then
$$
\rk_{}(W_{d_1}\otimes \cdots \otimes W_{d_k})\leq  2^{k-1}(d-2k+2).
$$
\end{corollary*}

Let us continue our running example illustrating the content of \Cref{theorem: bound rango}.
\begin{example}\label{example: step4}
Let $k=2$ and $\bfd=(3,3)$. We want to apply the procedure of \Cref{theorem: bound rango} to find a decomposition of $W_3\otimes W_3$. We know that $9W_3\otimes W_3=T+T^{\{2\}}$, where $T\in\spann(\mathcal C_{(3,3)})$ and $T\in\spann(\mathcal C_{(3,3)}^{\{2\}})$ are as in \Cref{example: step3}. Applying $\varphi^J$ we find
$$\varphi(T)=\alpha_0u^6+2^{-1}\binom{6}{4}u^2v^4+\alpha_6v^6,\qquad \varphi^{\{2\}}(T^{\{2\}})=-\alpha_0u^6+2^{-1}\binom{6}{4}u^2v^4-\alpha_6v^6.$$
The matrix $\Cat_3(\varphi(T))$ with respect to binomial basis of $\mathbb C[u,v]_6$ is

$$
\Cat_3(\varphi(T))=\begin{pNiceArray}{cccc}
\alpha_0 & 0 & 0 & 0\\
0 & 0 & 0 & 2^{-1}\\
0 & 0 & 2^{-1} & 0\\
0 & 2^{-1} & 0 & \alpha_6
\end{pNiceArray}
$$
and it is full rank if and only if $\alpha_0\neq 0$. By \Cref{prop: rango TJ} we know that to minimise the Waring rank of $\varphi(T)$ it is enough to take $\alpha_0\neq0$, so we can also assume $\alpha_6=0$. At this point we already know that $\rk(\varphi(T))=4$, but to find an explicit decomposition of $\varphi(T)$ we have to find a square-free element in the kernel of $\Cat_4(\varphi(T))$. We have
$$\Cat_4(\varphi(T))=\begin{pNiceArray}{ccccc}
\alpha_0 & 0 & 0 & 0 & 2^{-1}\\
0 & 0 & 0 & 2^{-1} & 0\\
0 & 0 & 2^{-1} & 0 & 0
\end{pNiceArray}, \qquad \ker(\Cat_4(\varphi(T)))=\langle\partial_u^4-2\alpha_0\partial_v^4,\partial_v^3\partial_u\rangle.$$
Since the only condition on $\alpha_0$ is $\alpha_0\neq 0$, we can take $\alpha_0=2^{-1}$ so that $\partial_u^4-\partial_v^4\in\ker(\Cat_4(\varphi(T)))$. Let $\omega\in\CC$ be a primitive 4th root of unity. Since the roots of the polynomial $\partial_u^4-\partial_v^4$ are $[\omega^i,1]$ for $i=0,1,2,3$, by Sylvester algorithm there exist $\lambda_1,\dots,\lambda_4\in\CC$ such that 
$$\varphi(T)=\lambda_1(u+v)^6+\lambda_2(\omega u+v)^6+\lambda_3(\omega^2 u+v)^6+\lambda_4(\omega^3 u+v)^6,$$
and the $\lambda_i$'s are the solution of the linear system
$$\begin{pmatrix}
    1 & \omega^2 & 1 & \omega^2\\
    1 & \omega & \omega^2 & \omega^3\\
    1 & 1 & 1 &1\\
    1 & \omega^3 & \omega^2 & \omega
\end{pmatrix}\begin{pmatrix}
    \lambda_1\\
    \lambda_2 \\
    \lambda_3\\
    \lambda_4
\end{pmatrix}=\begin{pmatrix}
    2^{-1}\\
    0\\
    0\\
    2^{-1}
\end{pmatrix}.$$ 
Denoting by $\eta\in\CC$ a 4th primitive root of $-1$ and repeating the same procedure with $J=\{2\}$, we find $\mu_1,\dots,\mu_4\in\CC$ such that
$$\varphi^{\{2\}}(T^{\{2\}})=\mu_1(-u+v)^6+\mu_2(\eta u+v)^6+\mu_3(\eta^2 u+v)^6+\mu_4(\eta^3 u+v)^6.$$
Finally, going back with $\varphi^{-1}$ and $(\varphi^{\{2\}})^{-1}$ one finds the desired decomposition.
\end{example}

\Cref{example: step4}, together with \Cref{example: step1}, \Cref{example: step2} and \Cref{example: step3}, concretely shows that, following our procedure, it is possible to construct an explicit decomposition of $W_{d_1}\otimes\cdots\otimes W_{d_k}$ having length as in the bound of \Cref{theorem: bound rango} as we now explain in detail.
\begin{remark}\label{remark: commento algo}
Let $\bfd\in\NN^k_{\geq 3}$. For $s\in[0,d]\setminus\{d-k\}$ we choose one $\bf i \in \A_{\bf d,s }$ and set $\varepsilon^J_{\bf i}=1$ for any $J\subseteq[2,k]$, and for $s=d-k$ and $\bfi=(d_1-1,\dots,d_k-1)$ we set $\varepsilon_\bfi^J=1$ for any $J\subseteq[2,k]$. In order to find a decomposition of $W_{d_1}\otimes\cdots\otimes W_{d_k}$ following our procedure, we have to first find $T^J\in\spann(\mathcal C_{\bfd}^J)$ such that 
$$\prod_{j=1}^kd_jW_{d_1}\otimes\cdots\otimes W_{d_k}=\sum_{J\subseteq[2,k]}T^J.$$
The coordinates of the $T^J$'s are of the form
$$T^J=(\varepsilon_{\bfi}^J\alpha_s^J)_{\substack{0\leq s\leq d\\ \bfi\in\A_{\bfd,s}}}$$
and, by \Cref{prop: rango TJ}, it is enough to take $\alpha_{k-2}^J\neq0$ and $\alpha_{d-k}^J\neq 0$ for every $J$ to minimize the length of the decomposition. The easiest possible choice is given by \Cref{prop: mortale}: we take $\alpha_s=0$ for $s\in[0,d]\setminus\{k-2,d-k\}$, $\alpha_{k-2}=(-1)^{|J|}$ and $\alpha_{d-k}=2^{k-1}$. Hence, when we compute $\varphi^J(T^J)$ we only have two possibilities:
$$\varphi^J(T^J)=\begin{cases}
     \binom{d}{k-2}u^{d-k+2}v^{k-2}+2^{k-1}\binom{d}{d-k}u^kv^{d-k},&\text{ if $|J|$ is even}\\
     -\binom{d}{k-2}u^{d-k+2}v^{k-2}+2^{k-1}\binom{d}{d-k}u^kv^{d-k},&\text{ otherwise.}
\end{cases}$$
As a consequence, it is enough to apply the Sylvester's algorithm to only two binary forms. Once a minimal decomposition of the two possible $\varphi^J(T^J)$'s is found it is enough to apply $(\varphi^J)^{-1}$ for all the $J\subseteq[2,k]$ to find a partially symmetric decomposition of the product of $W$-states. Note that this drastically reduce the computational cost. Indeed, the decompositions of the two binary forms and the application of $(\varphi^J)^{-1}$ are computationally cheap. The core of the advantage of this method is that the combinatorics of the $\varepsilon_\bfi^J$'s allows to avoid many computations. We summarized the procedure in \Cref{algo}.
\end{remark}

We now want to focus on the partially symmetric border rank of $W_{d_1}\otimes \cdots \otimes W_{d_k}$. We remark that when $d_i=d$ for all $i$, the border rank of the $k$-fold tensor Kronecker
product of $(W_d)^{\boxtimes k}$ is $2^k$ \cite[Theorem 2]{zuiddam2017note}. Moreover, \cite[Corollary 5.3]{CCGI25} proves that the smoothable rank of $(W_d)^{\otimes k}$ is $2^k$ and we expect the border rank of $ (W_d)^{\otimes k}$ to be also $2^k$. In the next result we confirm this expectation for arbitrary degrees.

\begin{theorem}\label{theorem: bound rango bordo}
Let $k\geq 2$, let $\bfd\in \NN_{\geq 3}^k$ and let $d=d_1+\cdots+d_k$.  Then
$$
\underline{\rk}_\bfd(W_{d_1}\otimes\cdots\otimes W_{d_k})= 2^{k}.
$$
\end{theorem}
\begin{proof}
We use the catalecticant bound for border rank of \cite[Corollary 5.5]{gal}. In order to do that we have to produce a partially symmetric flattening with rank at least $2^k$. Let us consider the following flattening of $W_{d_1}\otimes \cdots \otimes W_{d_k}$:
\[
F 
: \sym^2(\mathbb{C}^2) \otimes \dots \otimes \sym^2(\mathbb{C}^2) \rightarrow 
\sym^{d_1-2}(\mathbb{C}^2)^* \otimes \dots \otimes \sym^{d_k-2}(\mathbb{C}^2)^*.
\]
Let $\mathbb{C}[\alpha_{1,0},\alpha_{1,1},\dots,\alpha_{k,0},\alpha_{k,1}]$ be the ring of multigraded derivations on $\CC[x_{1,0},x_{1,1},\dots,x_{k,0},x_{k,1}]$ with $\alpha_{i,j}$ defined as the dual of $x_{i,j}$. By apolarity theory, the rank of the linear map $F$ 
is exactly the dimension of the degree $(2,\dots,2)$ component of the quotient of the multigraded ring $\mathbb{C}[\alpha_{1,0},\alpha_{1,1},\dots,\alpha_{k,0},\alpha_{k,1}]$ by the ideal $I_{d}$ apolar to the W-state, i.e. $I_{d}=(\alpha_{1,0}^2,\dots,\alpha_{k,0}^2,\alpha_{1,1}^{d_1},\dots,\alpha_{k,1}^{d_k})$. For $\varepsilon_i \in \{0,1\}$ the $2^k$ forms
$$F_{\varepsilon_1,\dots,\varepsilon_{k}}=\alpha_{1,0}^{\varepsilon_1}\alpha_{1,1}^{2-\varepsilon_1} \dots \alpha_{k,0}^{\varepsilon_k}\alpha_{k,1}^{2-\varepsilon_k}$$  are linear independent of multidegree $(2,\dots, 2)$, with $F_{\varepsilon_1,\dots,\varepsilon_{k}} \notin I_d$. This concludes the proof. 
\end{proof}

\begin{algorithm}[h]
\caption{A decomposition of $W_{d_1}\otimes \cdots \otimes W_{d_k}$ given by \Cref{theorem: bound rango}. }
\flushleft{
\textbf{Input:} }  The integers $k\geq 2$, $d_1,\dots,d_k\geq 3$.
\flushleft{
\textbf{Output:} A decomposition of $W_{d_1}\otimes \cdots \otimes W_{d_k}$ of length $2^{k-1}(d-2k+2)$.
}

\flushleft{
\textbf{Procedure:}
}\\
\begin{enumerate}[label=\it{Step \arabic*.},ref=\it{Step \arabic*.}]
\item 
   For $s\in[0,d]\setminus\{d-k\}$ choose one $\bf i \in \A_{\bf d,s }$ and set $\varepsilon^J_{\bf i}=1$ for any $J\subseteq[2,k]$. For $s=d-k$ and $\bfi=(d_1-1,\dots,d_k-1)$ set $\varepsilon_\bfi^J=1$ for any $J\subseteq[2,k]$.

    \item 
 Set $F=\binom{d}{k-2}u^{d-k+2}v^{k-2}+2^{k-1}\binom{d}{d-k}u^kv^{d-k}$ and $G=-\binom{d}{k-2}u^{d-k+2}v^{k-2}+2^{k-1}\binom{d}{d-k}u^kv^{d-k}$.

    \item 
        Apply Sylvester's algorithm to $F$ and $G$ to compute minimal decompositions $A_F$ of $F$ and $A_G$ of $G$. 
    
     \item
     \label{step 4}
    For $J\subseteq [2,k] $ do the following:
    \begin{itemize}
        \item if $|J|$ is even, apply $(\varphi^J)^{-1}$ to each element of $A_F$,
        \item if $|J|$ is odd, apply $(\varphi^J)^{-1}$ to each element of $A_G$.
    \end{itemize}

    \item[]  \textbf{Return} the sum of the preimages computed in \ref{step 4}
\end{enumerate}

\label{algo}
\end{algorithm}
\subsection{Further remarks on \Cref{theorem: bound rango} 
}\label{subsection: commenti sul bound}

As already emphasized, the general idea of bounding the rank of tensor product of $W$-states by looking at the containment of $W_{d_1}\otimes \cdots \otimes W_{d_k}$ in the span of the union of rational normal curves already appeared in \cite[Theorem 3.6]{BBCG}. However, our approach differs in two parts.

First, we choose a smaller number of rational curves, i.e. $2^{k-1}$ curves instead of $2^{k}$, all of multidegree $(1,\dots,1) $ in $(\PP^1)^{\times k}$. The downside of our choice of curves is that after the Segre-Veronese embedding they all result to be rational normal curves of degree $d=d_1+\cdots+d_k $, which is higher than the degrees appearing in  \cite[Theorem 3.6]{BBCG}.

Second, we are able to explicitly characterize the tensors $T^J \in \mathrm{span} (\mathcal{C}_{\textbf{d}}^J)$ giving the decomposition of $W_{d_1} \otimes \dots \otimes W_{d_k}$ seen as an element of $\spann\{ \cup_{J}\Cc_{\bf d}^J\} $ and we actually compute the Waring rank of each $T^J$.

We remark that the general idea of finding a curve $\mathcal C\subset\PP^1\times\cdots\times\PP^1$ such that $Z_k\subset \mathcal C$ and then estimating the rank by estimating the partially symmetric rank of $W_{d_1}\otimes\cdots \otimes W_{d_k}$ using the fact that $W_{d_1}\otimes\cdots \otimes W_{d_k}\in\spann (sv_{\bfd}(\mathcal C))$, can in principle be applied to any curve $\mathcal C$ containing $Z_k$. However, the only curves  for which we know how to effectively calculate the rank of a point in their span are rational normal curves. This is the reason why we choose $\mathcal C$ such that $sv_\bfd(\mathcal C)$ is a union of rational normal curves.

In the following, we clarify how our approach can essentially be considered optimal in the sense that, very reasonably, it is the best bound achievable by using the technique introduced in \cite{BBCG}.
\begin{remark}

Since $Z_k$ is a 2-symmetric scheme, see \cite[Definition 2.1]{CCGI25}, by \cite[Proposition 2.18]{CCGI25}, any smooth curve passing through its support intersects $Z_k$ in a 0-dimensional scheme of length 2. In particular, it follows that $\mathcal C$ must have at least $2^{k-1}$ irreducible components which are rational normal curves passing through the support of $Z_k$. We used curves having multidegree $(1,\dots,1)$, which are embedded in rational normal curves of degree $d$ via the Segre-Veronese map. \Cref{prop: rango TJ}, together with the fact that we have to consider at least $2^{k-1}$ curves, shows that our bound is the best achievable with these curves.

Since the generic rank of a tensor in the span of a rational normal curve increases with the degree of the curve, the most reasonable possibility to try to improve the bound is to use curves having some $0$'s in their multidegree. However, using the curves considered in \cite{BBCG}, which have some 0's in their multidegree, and computing the rank of the analogous of our $T^J$'s arising from this procedure, one finds a higher bound. For instance, using these curves for $W_3\otimes W_3$ gives $\rk_{3,3}(W_3\otimes W_3)\leq 9$, while our bound gives $\rk_{3,3}(W_3\otimes W_3)\leq 8$ and it is known that $\rk_{3,3}(W_3\otimes W_3)=8$.
\end{remark}
In principle, curves of higher multidegree could provide a better bound. However, obtaining an effective bound would require the corresponding points in the span of such curves to be very special, since the generic rank along such a high-degree curve is itself large, and it seems unlikely in our setting.

\bibliographystyle{alpha}
\bibliography{references.bib}
\Addresses

\end{document}